%% file: main.tex
\documentclass[12pt]{article}		
\usepackage[english]{babel} 		
\usepackage[utf8]{inputenc}			
\usepackage{amsmath, amssymb,amsthm}
\usepackage{graphicx}				
\usepackage{hyperref}
\usepackage{color}
\usepackage{tikz} 
\usepackage{appendix}
\usepackage{authblk}
\usepackage{subcaption} 
\usepackage{cite}
\usepackage{enumerate} 

\newtheorem{theorem}{Theorem}
\newtheorem{proposition}{Proposition}
\newtheorem{corollary}{Corollary}
\newtheorem{lemma}{Lemma}
\newtheorem{observation}{Observation}
\theoremstyle{definition}
\newtheorem*{definition}{Definition}
\newtheorem{example}{Example}

\renewcommand\det{\operatorname{Det}}
\newcommand\dist{\operatorname{Dist}}
\newcommand\aut{\operatorname{Aut}}

\title{Edge Determining Sets and Determining Index }
\author{Sean McAvoy and Sally Cockburn{\thanks{This research was supported by The Monica Odening '05 Student Internship and Research Fund in Mathematics.}}}
\affil{Hamilton College, Clinton NY}
\date{\today}

\begin{document}

\maketitle

\begin{abstract}
    A graph automorphism is a bijective mapping of the vertices that preserves adjacent vertices. A vertex determining set of a graph is a set of vertices such that the only automorphism that fixes those vertices is the identity. The size of a smallest such set is called the determining number, denoted $\det(G)$. The determining number is a parameter of the graph capturing its level of symmetry. We introduce the related concept of an edge determining set and determining index, $\det'(G)$. We prove that $\det'(G) \le \det(G) \le 2\det'(G)$ when $\det(G) \neq 1$ and show both bounds are sharp for infinite families of graphs. Further, we investigate properties of these new concepts, as well as provide the determining index for several families of graphs. 
\end{abstract}



\section{Introduction}

In this paper, we focus solely on finite, simple graphs, $G = (V, E)$,  using standard terminology and notation as can be found in \cite{chartrand2012first}.  
 A definition of particular interest for this paper is that an automorphism of a graph is a permutation of the set of vertices that respects vertex adjacency. The set of all automorphisms of a graph under the operation of composition forms a group, denoted by $\aut(G)$. 

Graph theorists often try to categorize graphs by their level of symmetry. Graphs with vertices or edges that are interchangeable display one such type of symmetry. More precisely, a graph is vertex-transitive if for all $u, v\in V(G)$, there exists $\phi \in \aut(G)$ such that $\phi(u) = v$. Edge-transitivity and arc-transitivity are defined similarly. Additionally, we will call a graph edge-flip-invariant if for all $\{u, v\} \in E(G)$, there exists $\phi \in \aut(G)$ such that $\phi(u)=v$ and $\phi(v) = u$. In Section~\ref{sec:comparing}, we will show that for connected graphs, edge-flip-invariance is a stronger condition than vertex-transitivity.

The determining number of a graph is another measure of its symmetry. The motivation is to find the least number of vertices that need to be fixed to break all of the symmetries of a graph. Intuitively, the more vertices that need to be fixed, the more symmetric the graph is.  

\begin{definition}
   A vertex subset $S$ of a graph is a {\it vertex determining set} if the only automorphism  $\phi \in \aut(G)$ that satisfies $\phi(u)=u$ for all $u \in S$ is the identity automorphism. The {\it determining number} of $G$, $\det(G)$, is the size of a smallest such set: $\det(G) = \min\{|S|\mid$ $S$ is a vertex determining set$\}$.
\end{definition}

Some authors use the term fixing number for this parameter, (see \cite{ERWIN20063244}, \cite{GRAPHSGROUPS}), but we will continue to refer to it as the determining number. Another way to break the symmetries of a graph is to assign a color to each vertex in such a way that otherwise interchangeable vertices can be distinguished. (Notice that this need not be a proper vertex coloring, in which adjacent vertices must be assigned different colors.)

\begin{definition} 
The {\it distinguishing number} of a graph $G$, $\dist(G)$, is the minimum number of vertex colors required so that the only automorphism $\phi \in \aut(G)$ that preserves all vertex colors is the identity automorphism.
\end{definition}

 Albertson and Boutin \cite{ALBERTSON2007} established a relationship between the determining number and distinguishing number.

\begin{proposition}\label{prop:distdet}\cite{ALBERTSON2007}
  For all graphs $G$, $\dist(G) \le \det(G) + 1$.
\end{proposition}

The idea  behind the proof is that the vertices of a minimum vertex determining set can be colored  with $\det(G)$ distinct colors, and all other vertices with another color. Since the only automorphism that fixes the vertices of the vertex determining set is the identity, this will be a distinguishing coloring with $\det(G)+1$ colors.

In 2013, Kalinowski and Pilśniak \cite{KALINOWSKI2015124} introduced the distinguishing index, based on coloring edges instead of vertices.

\begin{definition}
The {\it distinguishing index} of a graph $G$, $\dist'(G)$, is the minimum number of edge colors so that the only automorphism $\phi \in \aut(G)$ that preserves all edge colors is the identity automorphism.
\end{definition}

There has been further recent research on the distinguishing index (see for example  \cite{CHROMDIST2020}, \cite{AMC1852}, \cite{Lehner_2020}). Motivated by this work, in this paper we extend the concept of determining set from vertices to edges.  

 One complication is the following. An automorphism $\phi \in \aut(G)$ fixes an edge $e =\{u,v\} \in E(G)$ if $\phi(e) = e$, which means $\{\phi(u),\phi(v)\}=\{u,v\}$. In this case, $\phi$ either fixes the endvertices ($\phi(u) = u$ and $\phi(v) = v$) or switches the endvertices ($\phi(u) = v$ and $\phi(v) = u$). The following elementary observations will be useful in our investigation.

\begin{observation}\label{obs:adjedges}
Let $e_1, e_2$ be adjacent edges in a graph $G$. If $\phi \in \aut(G)$ fixes both $e_1$ and $e_2$, then $\phi$ fixes the endvertices of $e_1$ and $e_2$.
\end{observation}

\begin{observation}\label{obs:difdeg}
Let $e = \{u, v\}$ be an edge in graph $G$ such that $deg(u)\neq deg(v)$. If $\phi \in \aut(G)$ fixes $e$, then $\phi$ fixes the endvertices of $e$.
\end{observation} 

The open neighborhood $N(v)$ of a vertex $v$ is the set of all neighbors of $v$; the closed neighborhood of $v$ is $N[v]= \{v\} \cup N(v)$.  Distinct vertices $u$ and $v$ are called nonadjacent twins (respectively, adjacent twins) if $N(u) = N(v)$ (respectively, $N[u]= N[v]$).

\begin{observation}\label{obs:twins}
If $u, v$ are adjacent or nonadjacent twins in a graph $G$, then there exists $\phi \in \aut(G)$ such that $\phi(v) = u$, $\phi(v) = u$, and $\phi$ fixes all the other vertices.
\end{observation} 

 Any two isolated vertices in a graph $G$ will be nonadjacent twins, so there exists an automorphism switching them and leaving all other vertices fixed. This nontrivial automorphism fixes all edges of $G$.  Similarly, if there is a $K_2$ component of $G$, then its endvertices are adjacent twins. The automorphism switching the endvertices of this component and leaving all other vertices fixed is a nontrivial automorphism that fixes every edge in the graph. In both of these situations, it is not possible to define an edge equivalent of the determining number. With these considerations in mind, we make the following definition.

\begin{definition}
  Let $G$ be a graph with no more than one isolated vertex  and without $K_2$ as a component. An edge subset $T$  of $G$ is an {\it edge determining set} if the only $\phi \in \aut(G)$ that satisfies $\{\phi(u),\phi(v)\}=\{u,v\}$ for all $\{u,v\} \in T$, is the identity automorphism.
  The {\it determining index}, $\det'(G)$, is the size of the smallest such set: $\det'(G) = \min\{|T|\mid$ $T$ is an edge determining set$\}$.
\end{definition}

In other words, $T$ is an edge determining set if fixing every edge of $T$ fixes every vertex of the graph. If  $G$ is an asymmetric graph, meaning it has no nontrivial automorphisms, then both $\det(G) = 0$ and $\det'(G) = 0$. Due to the novelty of an edge determining set, previous authors have referred to a vertex determining set as simply a determining set. Henceforth, we will indicate clearly whether a determining set consists of edges or vertices.

The determining index of a disconnected graph is not as simple as the sum of the determining index of each of its connected components. For example, if $G$ is the union of two isomorphic, asymmetric, connected graphs, then the determining index is $0$ for each connected component, but $\det'(G)=1$, because there is a nontrivial automorphism switching the two components.

Erwin and Harary \cite{ERWIN20063244} observed the corresponding result for determining number of disconnected graphs. 

\begin{proposition}\label{obs:connected}\cite{ERWIN20063244}
Let $G = H_1 \cup ... \cup H_k$, where $H_i$ is a connected component of $G$. Let $P = \{H_i \mid \det(H_i) > 0\}$, $Q = \{H_i \mid \det(H_i) = 0\}$ and $R$ the 
set of isomorphism classes of graphs in $Q$. Then $\det(G) = |Q| - |R|+\sum_{H\in P} \det(H)$.
\end{proposition}

An analogous result holds for determining index.

\begin{lemma}\label{lemma:connected}
Let $G = H_1 \cup ... \cup H_k$, where $H_i$ is a connected component of $G$, at most one $H_i = K_1$, and no $H_i = K_2$. Let $P = \{C_i \mid \det'(H_i) > 0\}$, $Q = \{H_i \mid \det'(H_i) = 0\}$ and $R$ the set of isomorphism classes of graphs in $Q$. 
Then $\det'(G) = |Q| - |R|+\sum_{H\in P} \det'(H)$.
\end{lemma}

\begin{proof}
If $H_i \in P$, then clearly we need to fix $\det'(H_i)$ edges to fix all vertices of $H_i$. 

If $H_i \in Q$, then for all $H_j \in Q$ such that $i \neq j$ and $H_i \cong H_j$, there exists an automorphism of $G$ that interchanges the vertices of these two components. Let $k$ be the number of connected components in the isomorphism class of $H_i$. We need to fix at least $k-1$ of these components. Fix one edge of each of the components in the isomorphism class of $H_i$ except one. Then the vertices of all the components in the isomorphism class of $H_i$ are fixed. We need to do this procedure for each distinct isomorphism class in $Q$. Thus, $\det'(G) = |Q| - |R| + \sum_{H\in P} \det'(H)$.
\end{proof}

We now present results for the determining number of several families of graphs for comparison to our later results on the determining index. Greenfield \cite{GREENFIELD2011} proved the following. 

\begin{proposition}\label{prop:detfamilies}\cite{GREENFIELD2011}
  For $n>1$, $\det(P_n) = 1$, $\det(C_n) = 2$, and $\det(K_n) = n-1$. For $n\geq m > 1$, $\det(K_{n,m}) = n+m-2$.
\end{proposition}

The line graph of $G$, $L(G)$, is a graph such that $V(L(G)) = E(G)$, with vertices $e_1 = \{u,v\}$ and $e_2 = \{x,y\}$ in $L(G)$ being adjacent if the corresponding edges are adjacent in $G$. It is easy to verify that an automorphism $\alpha$ of $G$ induces an automorphism $\alpha'$ of $L(G)$ in the obvious way; for all $e = \{u, v \} \in V(L(G))$, define $\alpha'(e) = \{\alpha(u), \alpha(v)\}$. Rodriguez \cite{Rodriguez2014AUTOMORPHISMGO} showed that except for the three cases shown in Figure~\ref{fig:exceptionalCases}, these induced automorphisms make up the entire automorphism group of $L(G)$.

\begin{proposition}\label{prop:aut}\cite{Rodriguez2014AUTOMORPHISMGO}
  Let $G$ be a connected graph such that $|V(G)| \geq 3$. If $G\notin \{G_1, G_2, G_3\}$, then $\aut(G) \cong \aut(L(G))$.
\end{proposition}

\input{figure1}\label{fig:exceptionalCases}

 In 2018, Alikhani and Soltani established a relationship between the distinguishing index of a graph and the distinguishing number of the corresponding line graph.

\begin{proposition}\label{prop:distdist'}
\cite{ALIKHANI2018}
 If $G$ is a connected graph such that $|V(G)|\geq3$ and $G \neq G_2$, then $\dist'(G) = \dist(L(G))$.
\end{proposition}

Using Proposition~\ref{prop:aut}, we establish a similar relationship as Proposition~\ref{prop:distdist'} between the determining index of a graph and the determining number of the corresponding line graph.

\begin{theorem}\label{thm:det'LG}
  Let $G$ be a connected graph such that $|V(G)| \geq 3$. Then $\det'(G)=\det(L(G))$ if and only if $G\notin \{G_1, G_2, G_3\}$.
\end{theorem}

\begin{proof}
    
    Assume $G\notin \{G_1, G_2, G_3\}$.
  We will first show $\det'(G)\geq \det(L(G))$. Let $T$ be a minimum edge determining set of $G$ so that $\det'(G) = |T|$. Then $T$ is a set of vertices in $L(G).$ Let $\alpha$ be an automorphism of $L(G)$ such that $\alpha(v) = v$ for all $v \in T$. By Proposition~\ref{prop:aut}, there exists an isomorphism $\phi:\aut(G) \rightarrow \aut(L(G))$, so there exists $\sigma\in \aut(G)$ such that
  %that induces $\alpha$. In other words, 
  $\phi(\sigma) =\alpha$. Hence, for all edges $e = \{u,v\} \in E(G)$, $\alpha(e) = \{\sigma(u),\sigma(v)\}$. But since $\alpha(e)=e$ for all $e \in T$, $\{\sigma(u),\sigma(v)\} = \{u,v\}$ for all $\{u,v\} \in T$. Since $T$ is an edge determining set in $G$, $\sigma$ is the identity in $\aut(G)$. Hence, $\alpha$ is the identity in $\aut(L(G))$. By definition, $T$ is a vertex determining set in $L(G)$. Since $T$ is of minimum size as an edge determining set, $\det'(G)\geq \det(L(G))$.

We will now show $\det(L(G))\geq \det'(G)$. Let $S$ be a minimum vertex determining set of $L(G)$ so that $\det(L(G)) = |S|$. Then $S$ is a set of edges in $G$. Let $\tau \in \aut(G)$ such that $\{\tau(u),\tau(v)\} = \{u, v\}$ for all $\{u, v \}\in S$. By Proposition~\ref{prop:aut},  
%there exists a bijection $\phi:\aut(L(G)) \rightarrow \aut(G)$. Thus 
there is a unique $\beta \in \aut(L(G))$ such that $\phi(\tau) = \beta$. Then for all $e = \{u, v \}\in S$, $\beta(e)=\{\tau(u),\tau(v)\} = \{u, v\}=e$. Since $S$ is a vertex determining set in $L(G)$, $\beta$ is the identity automorphism of $L(G)$. Since $\phi:\aut(G) \cong \aut(L(G))$ is an isomorphism, the only element of $\ker(\phi)$ is the identity automorphism  of $G$. Thus $\tau$ is the identity and so $S$ is an edge determining set of $G$. Since $S$ is of minimum size as a vertex determining set of $L(G)$, $\det'(G) \le \det(L(G))$. Hence $\det'(G) =\det(L(G))$.

Conversely, assume $G\in \{G_1, G_2, G_3\}$. By inspection, we have $\det'(G_1) = \det'(G_2)=1$, and $\det'(G_3)=2$. For the determining number of their line graphs, shown in Figure~\ref{fig:linegraphs}, we have $\det(L(G_1)) = 2$, $\det(L(G_2))=2$, and $\det(L(G_3))=3$. Hence, if $G \in \{G_1, G_2,  G_3\}$, then $\det'(G) \neq \det(L(G))$.

\end{proof}

\input{line_graphs}

With the connection between $G$ and $L(G)$, we now have a tool to find the determining index, if the determining number of the line graph is known. 

\begin{theorem}\label{thm:PathsCyclesStars}
  For all $n > 2$,  $\det'(P_n)=1$ and for all $n > 1$,   $\det'(C_n)=2$, and $\det'(K_{1,n})=n-1$ . 
\end{theorem}

\begin{proof}
    Note that $P_{n-1} = L(P_n)$, $C_n = L(C_n)$, and $K_{n} = L(K_{1,n})$.
    Thus, the result follows from Proposition~\ref{prop:detfamilies} and Theorem~\ref{thm:det'LG}.
\end{proof}

We can  easily establish a similar relationship between the distinguishing index and the determining index to that between the distinguishing number and determining number in Proposition~\ref{prop:distdet}.

\begin{theorem}
  Let $G$ be a connected graph such that $|V(G)| \ge 3$. Then $\dist'(G)\leq \det'(G) + 1$.
\end{theorem}

\begin{proof}
We first cover the special cases by inspection.
If $G = G_1$, we have $\dist'(G_1)=1 \leq 2 =\det'(G_1)+1$. If $G = G_2$, we have $\dist'(G_2)=1 \leq 2 = \det'(G_2) + 1$.
If $G = G_3$, we have $\dist'(G_3)=3 = \det'(G_3)+1$.
If $G\neq G_1$, $G_2$, or $G_3$, then
by Proposition~\ref{prop:distdet} and Proposition~\ref{prop:distdist'}, 
\[\dist'(G) = \dist(L(G)) \leq \det(L(G)) + 1 = \det'(G)+1.\]
\end{proof}

\section{Comparing Determining Number and  Determining Index}\label{sec:comparing}

In this section, we discuss general relationships between the determining index and determining number. 
One important difference between these parameters is that while $\det(G)=\det(\overline{G})$, where $\overline G$ is the complement of $G$, the same is not always true for the determining index. This can be because the determining index is undefined for the complement.

\begin{example}
  By Theorem~\ref{thm:PathsCyclesStars}, $\det'(C_3) = 2$. However, $\det'(\overline{C_3})$ is undefined as the number of isolated vertices is $3$.
\end{example}

However, this is not always a result of $\det'(G)$ being undefined.

\begin{example}
  By Theorem~\ref{thm:PathsCyclesStars}, $\det'(K_{1,5}) = 4$. However,  $\overline{K_{1,5}} = K_5 \cup K_1$,  so $\det'(\overline{K_{1,5}})= 3$. With only one isolated vertex, the edge determining index of $\overline{K_{1,5}} = K_5 \cup K_1$ is still defined.
\end{example}

By Theorem~\ref{thm:det'LG}, knowing when $\det'(G)=\det'(\overline{G})$ can be viewed as a problem of knowing when $\det(L(G))=
\det(L(\overline{G}))$. 
Alternatively, we can ask when $\det(G) = \det'(G)$ and $\det(\overline{G})=\det'(\overline{G})$ as $\det(G)= \det(\overline{G})$. If $G$ is self-complementary, then the equality clearly holds. 

\medskip

In  \cite{KALINOWSKI2015124}, Kalinowski and Pilśniak show that $\dist'(G) \le \dist(G) + 1$ for any graph $G$. It is natural to expect that there is a similar relationship between the determining index and the determining number of a graph. Before establishing a connection, 
we present some examples displaying the subtleties of constructing an edge determining set from a given vertex determining set. For a given vertex determining set $S$, it seems reasonable to believe that we can construct a corresponding edge determining set by picking one edge incident to each vertex in $S$. In Figure~\ref{fig:P_4a}, the red vertex constitutes a determining set, but the incident blue edge fails to constitute an edge determining set. However, Figure~\ref{fig:P_4b} shows that the other incident edge does constitute an edge determining set. In this example, $\det(P_4) = \det'(P_4) = 1$.

\input{figure2}

The example of $P_4$ fails to highlight all possible difficulties in selecting edges incident to vertices. While the edge we chose mattered, there was an obvious one. The endvertices of $\{3,4\}$ have different degrees, so fixing that edge will fix its endvertices by Observation~\ref{obs:difdeg}.  The choice is less obvious in a graph with more symmetry, such as the `envelope' graph $H$ in Figure~\ref{fig:env}, which is both vertex-transitive and edge-flip-invariant.

For this example, it is easily verified that $\{3, 5\}$ is a minimum vertex determining set, so $\det(H) =  2$. In Figure~\ref{fig:enva}, the blue edges fail to constitute an edge determining set because the reflection across a central vertical line in the drawing is an automorphism that flips edges $\{3,4\}$ and $\{5,6\}$ (and $\{1,2\}$, in fact). 
However, Figure~\ref{fig:envb} indicates that a different choice of one incident edge per vertex can still produce an edge determining set. However, trying to use only the one edge between vertices $3$ and $5$ won't work because there is an automorphism that flips this edge (as well as the edge $\{4, 6\}$). The same is true of any singleton edge set, because $H$ is edge-flip-invariant. Thus $\det'(H) = 2$. With these nuances in mind, we prove the following. Recall that the distance $d(u,v)$  between vertices $u$ and $v$ is the minimum number of edges in a $u-v$ path; a path of minimum length is called a $u-v$ geodesic. As is shown in \cite{chartrand2012first},  any subpath of a geodesic is also a geodesic.

\input{figure3}\label{fig:envelope}

\begin{theorem}\label{thm:lowerbound}
  Let $G$ be a connected graph such that $|V(G)| \ge 3$. If $\det(G) \neq 1$, then $\det'(G)\leq \det(G)$.
\end{theorem}

\begin{proof}

Let $S = \{s_1, ..., s_k\}$ be a minimum vertex determining set of $G$ such that $k > 1$. 

First suppose all the vertices of  $S$ are pairwise adjacent. 
Then there exists a path $P$ such that $V(P) = S$. Let $T$ be the edges of $P$, so that $|T| = k-1$. If $k = 2$, there exists an edge $e$ adjacent to $\{s_1, s_2\}$, since $G$ is connected with $|V(G)|\ge 3$.
In this case, we add $e$ to $T$ to create an edge set $T'$ of size $k=2$.
Assume $\phi \in \aut(G)$ fixes the edges of $T$.  Let $\{u,v\}, \{v,w\} \in T$ be adjacent edges. 
By Observation~\ref{obs:adjedges}, if two adjacent edges are fixed, then the endvertices of those edges are fixed. Thus $\phi(u)=u$, $\phi(v)=v$, $\phi(w)=w$. Working our way along the adjacent edges of path $P$, we can conclude that all the endvertices of the edges of $T$ are fixed. Therefore, $\phi(s)=s$ for all $s \in S$ and so $\det'(G) \le \det(G) = k$. \medskip

Now suppose not all of the vertices are pairwise adjacent. Then by renumbering the vertices in $S$ if necessary, we can assume $d(s_1, s_2) \ge 2$.  Let $P = \{s_1, u_1,..., u_2, s_2\}$ be a $s_1 - s_2$ geodesic. Note that it is possible for $u_1=u_2$. We will inductively construct an edge determining set  consisting of edges incident to vertices of $S$. We start with the two edges $e_1 =\{s_1, u_1\}$ and $e_2 =\{s_2, u_2\}$. Let $T_2 = \{e_1, e_2\}$. Assume $\phi \in \aut(G)$ fixes these two edges. If $u_1=u_2$, then the edges are adjacent, so by Observation~\ref{obs:adjedges}, $\phi(s_1) = s_1$ and $\phi(s_2) = s_2$. Otherwise, the edges are nonadjacent. In this case, suppose that $\phi$   flips $e_1$ but fixes the vertices of $e_2$. Then $\phi(s_1)=u_1$, $\phi(u_1)=s_1$, and $\phi(s_2)=s_2$. Since automorphisms are distance-preserving, $d(s_1, s_2)=d(\phi(s_1),\phi(s_2))=d(u_1,s_2)$. This contradicts our assumption that that $P$ is a geodesic. Now suppose $\phi$ flips both $e_1$ and $e_2$, so that $\phi(s_1)=u_1$, $\phi(u_1)=s_1$, and $\phi(s_2)=u_2$. Again, since automorphisms are distance-preserving, $d(s_1, s_2)=d(\phi(s_1),\phi(s_2))=d(u_1,u_2)$. This also contradicts our assumption that $P$ is a geodesic. Thus, if $\phi$ fixes the edges of $T_2$, then $\phi(s_1)=s_1$, and $\phi(s_2)=s_2$.   \medskip

Next let $2<i\leq k$ and let $T_{i-1}= \{e_1, e_2, \dots, e_{i-1}\}$ be a set of edges of the form $e_j = \{ s_j, u_j\}$ with $d(s_1, u_j) < d(s_1, s_j)$. Assume that any automorphism fixing every edge of $T_{i-1}$ also fixes the vertices $s_1, \dots, s_{i-1}$.
Since $G$ is connected, there exists a path from $s_1$ to $s_i$. Let $e_i = \{s_i,u_i\}$ such that  $u_i$ is adjacent to $s_i$ on an $s_1 - s_i$ geodesic. Thus $d(s_1, u_i) < d(s_1, s_i)$. Note that it is possible that $u_i = s_j$ for some $j \in \{1, \dots, i-1\}$. However, in that case it is not possible that $e_i = \{s_i, u_i\} = \{s_i, s_j\}$ is already in $T_{i-1}$, because this would imply $s_i = u_j$, which in turn would imply
\[
d(s_1, s_j) = d(s_1, u_i) < d(s_1, s_i) = d(s_1, u_j) < d(s_1, s_j). 
\]
Let $T_i = T_{i-1} \cup \{e_i\}$.

Let $\phi \in \aut(G)$ be an automorphism of $G$ that fixes the edges of $T_i$. Then $\phi$ must fix the edges of $T_{i-1}$ and so by assumption, $\phi$ fixes $s_1, \dots , s_i$.
If $u_i = s_j$ for some $j \in \{1, \dots, i-1\}$, then since $\phi(u_i) =\phi(s_j) = s_j = u_i$, it must be the case that $\phi(s_i) = s_i$ also.
Otherwise, assume that $\phi$ flips edge $e_i$; that is, 
assume that $\phi(s_i)=u_i$ and $\phi(u_i)=s_i$. Then since automorphisms are distance-preserving, $d(s_1,s_i)=d(\phi(s_1),\phi(s_i)) = d(s_1, u_i)< d(s_1, s_i)$, a contradiction. Thus, $\phi(s_i)=s_i$.

Finally, let $T =T_k$. If $\phi \in \aut(G)$ fixes every edge in $T$, then $\phi(s_i)=s_i$ for all $s_i \in S$. Since $S$ is a determining set,
$\phi$ is the identity automorphism.
Therefore, by definition $T$ is an edge determining set. For each vertex in $S$, we added a distinct edge when constructing $T$. Therefore, $|T| = |S|$. Thus, $\det'(G)\leq \det(G)$. 
\end{proof}

\begin{theorem}\label{thm:upperbound}
  Let $G$ be a connected graph such that $|V(G)| \ge 3$. Then $\det(G)\leq 2\det'(G)$.
\end{theorem}

\begin{proof}

Let $T$ be a minimum edge determining set of $G$ and let $S$ be the set of endvertices of the edges in $T$. If the edges of $T$ are all nonadjacent, then $|S|= 2|T|$; allowing for the possibility that some vertices are endvertices of two or more edges of $T$ means $|S|
\le 2|T| = 2 \det'(G)$.
Let $\phi \in \aut(G)$ such that $\phi$ fixes the vertices of $S$. Let $\{u,w\} \in T$. Then $\{\phi(u),\phi(w)\} = \{u,w\}$, since $u, w \in S$. Since $T$ is an edge determining set, $\phi$ is the identity automorphism. Therefore, by definition, $S$ is a vertex determining set.
Thus, $\det(G) \le 2\det'(G)$.
\end{proof}

If $\det(G) \neq 1$, then together Theorem~\ref{thm:lowerbound} and Theorem~\ref{thm:upperbound} give us the inequality $\det'(G)\leq \det(G) \leq 2\det'(G)$. We can algebraically rewrite this inequality to give bounds on $\det'(G)$ in terms of $\det(G)$.

\begin{corollary}\label{cor:upperlower}
Let $G$ be a connected graph such that $|V(G)| \ge 3$. If $\det(G)\neq 1$, then $\frac{1}{2}\det(G)\leq \det'(G) \leq \det(G)$.
\end{corollary}

We have so far avoided the case when $\det(G) = 1$. In order to make claims about $\det'(G)$ in this case, we need to establish results about edge-flip-invariant graphs. Recall that $G$ is edge-flip-invariant if for all $\{u, v\} \in E(G)$, there exists $\phi \in \aut(G)$ such that $\phi(u)=v$ and $\phi(v) = u$. We now introduce the following definition. 

\begin{definition}
  A vertex $v \in V(G)$ has {\it neighbor-swapping property} if for all $u \in N(v)$, there exists $\phi \in \aut(G)$ such that $\phi(v) = u$ and $\phi(u)=v$.
\end{definition}

Clearly, all the vertices of $K_n$ have this property. Many other symmetric graphs have vertices with this property such as $H$ in Figure~\ref{fig:env}, the hypercubes $Q_n$ and cycles $C_n$.

\begin{theorem}\label{thm:neigh}
For a connected graph $G$, the following are equivalent:
\begin{enumerate}[(a)]
    \item there exists $v \in V(G)$ that has neighbor-swapping property;
    \item every $v \in V(G)$ has neighbor-swapping property;
    \item $G$ is edge-flip-invariant.
\end{enumerate}
In this case, $G$ is vertex-transitive.
\end{theorem}

\begin{proof}
Assume there exists $v \in V(G)$ such that $v$ has neighbor-swapping property.
Let $x \in N(v)$ and $y \in N(x)$ such that $y \neq v$, as shown in Figure~\ref{fig:neighborneighbor}. By assumption, there exists $\alpha \in \aut(G)$ such that $\alpha(x)=v$ and $\alpha(v)=x$. Then $\alpha(y) \in N(\alpha(x))=N(v)$. Hence, there exists $z \in N(v)$ such that $\alpha(y)=z$. Since $v$ has neighbor-swapping property, there exists $\beta \in \aut(G)$ such that $\beta(z)=v$ and $\beta(v)=z$. Define $\sigma: V(G) \rightarrow V(G)$ by $\sigma = \alpha^{-1} \beta \alpha$. Then $\sigma(x) = y$ and $\sigma(y) = x$. Hence, for all $y \in N(x)$, there is an automorphism switching $x$ and $y$.
Hence, $x$ has neighbor-swapping property.

\input{figure_neighbor_ex}

Now let $w \in V(G)$. Since $G$ is connected, there exists a $v - w $ path  $P = (v, u_1, ..., u_k, w)$. By the argument above,  $u_1$ has neighbor-swapping property. Similarly, since $u_2 \in N(u_1)$, $u_2$ has neighbor-swapping property. Continuing along the edges of the path, $w$ must have neighbor-swapping property. Thus, all the vertices of $G$ are interchangeable with their neighbors. By definition, $G$ is edge-flip-invariant.

Now assume $G$ is edge-flip-invariant. Let $v \in V(G)$ and $u \in N(v)$. By definition, there exists $\phi \in \aut(G)$ such that $\phi(u) = v$ and $\phi(v) = u$. Thus, there exists $v \in V(G)$ such that $v$ has the neighbor-swapping property.

Next we show that if $G$ is connected and edge-flip-invariant, then $G$ is vertex-transitive. Let $v, w \in V(G)$. Since $G$ is connected, there exists a $v - w$ path, $P = (v= u_1, u_2, ..., u_k= w)$. Since $G$ is edge-flip-invariant, there exists $\phi_i \in \aut(G)$ such that $\phi_i(u_i)=u_{i+1}$ and $\phi_i(u_{i+1}) = u_i$ for all $1 \le i \le k$. Define $\phi: V(G) \rightarrow V(G)$ by $\phi = \phi_{k-1} \circ ... \circ \phi_1$. Then $\phi(u_1) = u_k$. Thus, $G$ is vertex-transitive.
\end{proof}

Notably, not all vertex-transitive graphs are edge-flip-invariant, such as the Holt graph.

\begin{example} The Holt graph is the smallest graph that is vertex-transitive, edge-transitive, but not arc-transitive \cite{HoltGraph}.  Since it is edge-transitive, if there is an automorphism that flips any one edge, then it can be composed with other automorphisms to create an automorphism taking any arc to any other arc. This contradicts the fact that it is not arc-transitive. \end{example}

Now we can establish the determining index when $\det(G) = 1$. In this case, it is possible that $\det'(G) > \det(G)$.

\begin{corollary}\label{cor:det1}
Let $G$ be a connected graph such that $|V(G)| \ge 3$ and $\det(G) = 1$. 
Then
\[
\det'(G) = \begin{cases} 2, \quad & \text{ if $G$ is edge-flip-invariant},\\
1, &\text{otherwise}.
\end{cases}
\]
\end{corollary}

\begin{proof}
Let $S = \{v\}$ be a minimum vertex determining set of $G$. 

If $G$ is not edge-flip-invariant, then by Theorem~\ref{thm:neigh}, no vertex, including $v$, has neighbor-swapping property. Thus, there exists $u\in N(v)$ such that no automorphism interchanges $u$ and $v$. Let $T = \{\{u,v\}\}$.  Let $\phi \in \aut(G)$ such that $\{\phi(u), \phi(v)\}=\{u,v\}$. Then $\phi(u) = u$ and $\phi(v)=v$. Since the vertex of $S$ is fixed, $\phi(w) = w$ for all $w \in V(G)$. Hence, $T$ is a minimum edge determining set and $\det'(G) = 1$.

Now assume that $G$ is edge-flip-invariant. By definition, there is a nontrivial automorphism that fixes any given edge. 
Then clearly $\det'(G) > 1$. Since $G$ is connected with $|V(G)| \ge 3$, we can find two adjacent edges $e_1$ and $e_2$ such that $v \in e_1 \cup e_2$.
Let $T= \{e_1, e_2\}$ and assume that $\phi \in \aut(G)$ fixes the edges of $T$. Then by Observation~\ref{obs:adjedges}, the endvertices of the edges are fixed.
Since the vertex of $S$ is fixed, $\phi(w) = w$ for all $w \in V(G)$. Hence, $T$ is a minimum determining set and $\det'(G)=2$.
\end{proof}

\input{figure4}

The following examples to show that there exist graphs for both cases. 

\begin{example} By Proposition~\ref{prop:detfamilies}, $\det(P_4)=1$, and clearly $P_4$ is not edge-flip-invariant. Hence, in accordance with Theorem~\ref{thm:PathsCyclesStars} and as seen earlier, $\det'(P_4)=1$.\end{example}

\begin{example} In \cite{math9020166}, Brooks et al.  show that any automorphism fixing any vertex of the graph $G_4$ shown in Figure~\ref{fig:BrooksGraph} will fix all other vertices of the graph. Therefore, $\det(G_4) = 1$. 
We can use Theorem~\ref{thm:neigh} to show the graph is edge-flip-invariant. Looking at the edges incident to $0$, we list the following automorphisms as permutations of the vertices. These are reflections across the dashed lines in the drawing.   \medskip

$\alpha = (0\ 15)(7\ 17)(16\ 8)(9\ 6)(1\ 14)(10\ 5)(2\ 13)(11\ 4)(3\ 12)$

$\beta = (0\ 17)(9\ 8)(16\ 1)(7\ 10)(15\ 2)(11\ 6)(14\ 3)(12\ 5)(13\ 4)$

$\gamma = (0\ 9)(1\ 17)(8\ 10)(16\ 2)(7\ 11)(15\ 3)(6\ 12)(14\ 4)(5\ 13)$. \medskip

Thus, $0$ has neighbor-swapping property. By Theorem~\ref{thm:neigh}, the graph is edge-flip-invariant. By Corollary~\ref{cor:det1}, $\det'(G_4) = 2$.\medskip
\end{example}

\section{The Determining Index for Some  Families of Graphs}

In this section, we find the determining index for several different families of graphs. This can give an idea of the differences between the determining index and the determining number as well as the sharpness of the bounds in Corollary \ref{cor:upperlower}. Recall that $\det(K_{n,m})=n+m-2$ for $n,m>1$ and $\det(K_n) = n-1$. 

\begin{theorem}
  For $n \ge m > 1$,
  \[
  \det'(K_{n, m}) = \begin{cases} n-1, \quad &\text{ if } n \neq m,\\
  n, & \text{ if } n = m.
  \end{cases}
  \]
  
\end{theorem}

\begin{proof}
    Let $U = \{u_1, u_2,..., u_n\}$ and $V = \{v_1, v_2,..., v_m\}$ be the partites of $K_{n,m}$. Notice that all the vertices in $U$ are pairwise nonadjacent twins, as are all the vertices in $V$.
    \smallskip
    
    Case 1: Assume $n> m$. By Observation~\ref{obs:twins}, any edge determining set $T$ must have 
    at least $n-1$ edges to cover $n-1$ vertices in $U$.  Let \[T = \{\{u_i,v_i\} \mid 1 \le i \le m\} \cup \{ \{u_j, v_m\} \mid m < j < n\}.\]
    
    Let $\phi \in \aut(G)$ and assume $\phi$ fixes the edges of $T$. Since $n > m$, the degrees of the vertices in $U$ and $V$ are different. By Observation~\ref{obs:difdeg}, $\phi$ must fix the endvertices of every edge in $T$. We have fixed $u_i$ for $1 \le i < n$, and hence, $\phi(u_n)=u_n$. Since $|T| = n-1$, $\det'(G) = n - 1$. \medskip

    Case 2: Assume $n = m > 1$. Assume there exists an edge determining set of size $n-1$. The edges of an edge determining set must cover $n-1$ vertices in $U$ and $n-1$ vertices in $V$. Hence, each edge in an edge determining set of size $n-1$ would cover a distinct vertex in $V$ and a distinct vertex in $U$. By renumbering the vertices if necessary, we can assume that $T =\{ \{u_i, v_i\} \mid 1 \le i \le n-1\}$. Now let $\phi$ be the nontrivial automorphism  $\phi(u_i) = v_i$ and $\phi(v_i) = u_i$ for all $1 \le i \le n-1$. Therefore, there does not exist an edge determining set of size $n-1$. 
    
    Let $T = \{\{u_i,v_i\}\mid$ for $1 \le i < n\} \cup \{u_1,v_2\}$.
    Assume $\phi \in \aut(G)$ fixes the edges of $T$. By Observation~\ref{obs:adjedges}, $\phi(u_1)=u_1$, $\phi(v_1)=v_1$, and $\phi(v_2)=v_2$. If $\phi$ fixes one vertex in $U$, then $\phi$ preserves the partite sets $U$ and $V$ setwise. Thus,
    if $\{\phi(u_i), \phi(v_i)\} = \{u_i, v_i\}$ for $1 \le i < n$, 
    then $\phi(u_i)=u_i$ and $\phi(v_i)=v_i$. We have fixed $u_i$ for $1 \le i < n$, and hence, $\phi(u_n)=u_n$. Similarly, we have fixed $v_i$ for $1 \le i < n$, and hence, $\phi(v_n)=v_n$.
    Since $|T| = n$, $\det'(K_{n,n}) = n$.
\end{proof}

\begin{theorem}
  For $n > 2$, $\det'(K_n)=\lfloor \frac{2n}{3} \rfloor$.
\end{theorem}

\begin{proof}
    Let $G = K_n$ such that $n > 2$. Note all vertices are pairwise adjacent twins. Hence, an edge determining set of $G$ must cover at least $n-1$ vertices.
    Further, if $T$ is an edge determining set and $\{u,v\} \in T$, then $\{v,w\} \in T$ or $\{u,w\} \in T$ for some other $w \in V(G)$. Otherwise, there exists an automorphism that switches $u$ and $v$ and leaves the other vertices fixed. 
    
    Each pair of adjacent edges covers $3$ distinct vertices. If each pair of adjacent edges in anedge determining set covers $3$ distinct vertices and $n = 0 \pmod 3$ or $n = 1 \pmod 3$, then these edges cover at least $n-1$ vertices. If each pair of adjacent edges in an edge determining set covers $3$ distinct vertices and $n = 2 \pmod 3$, then $2$ twin vertices would not be covered. Hence, the edge set would not be determining. Therefore, $\det'(G)\geq 2\lfloor \frac{n}{3} \rfloor$ if $n = 0 \pmod 3$ or $n = 1 \pmod 3$. Otherwise, $\det'(G)> 2\lfloor \frac{n}{3} \rfloor$.
    %SC - This write up is much better!!
    
    Let $V(G) = \{v_1,...,v_n\}$ and let \[T=\{ \{v_{i-1},v_i\}, \{v_i,v_{i+1}\} \mid  i = 2 \pmod 3,\, 0<i<n, \}.\]
     Then $|T| = 2\lfloor \frac{n}{3} \rfloor$. \medskip

    Case 1: Suppose $n = 0 \pmod 3$. Then $T$ covers all the vertices of $G$. Let $\phi \in \aut(G)$ and assume $\phi$ fixes the edges in $T$. 
    By Observation~\ref{obs:adjedges}, $\phi(v_i)=v_i$, $\phi(v_{i-1}) = v_{i-1}$, and $\phi(v_{i+1})=v_{i+1}$ for $i = 2 \pmod 3$, $0<i<n$. Thus, $T$ is an edge determining set since $T$ covers the vertices of $G$. Note $|T|=2\lfloor \frac{n}{3} \rfloor$, so $T$ is a minimum edge determining set. Since $n = 0 \pmod 3$, $\det'(G) = \lfloor \frac{2n}{3} \rfloor$.\medskip
    
    Case 2: Suppose $n = 1 \pmod 3$. Then $T$ covers all of the vertices of $G$ except one. Let $\phi \in \aut(G)$ and assume $\phi$ fixes the edges in $T$. 
    By Observation~\ref{obs:adjedges}, $\phi(v_i)=v_i$, $\phi(v_{i-1}) = v_{i-1}$ and $\phi(v_{i+1})=v_{i+1}$ for $i = 2 \pmod 3$, $0<i<n$. Thus, $T$ is an edge determining set, since $T$ covers every vertex of $G$ except one. Note $|T|=2 \lfloor \frac{n}{3} \rfloor$, so $T$ is a minimum edge determining set. Since $n = 1 \pmod 3$, $\lfloor \frac{2n}{3} \rfloor = \lfloor \frac{2(n-1)}{3} + \frac{2}{3} \rfloor = 2 \lfloor \frac{n-1}{3}  \rfloor=2 \lfloor \frac{n}{3}  \rfloor$.\medskip
    
    Case 3: Suppose $n = 2 \pmod 3$. Then there exist two vertices, $v_{n-1}$ and $v_n$, not covered by an edge in $T$.  Thus, $T$ is not an edge determining set. 
    Add edge $\{v_n, v_1\}$ to $T$. Let $\phi \in \aut(G)$ and assume $\phi$ fixes the edges in $T$. 
    By Observation~\ref{obs:adjedges}, $\phi(v_i)=v_i$, $\phi(v_{i-1}) = v_{i-1}$ and $\phi(v_{i+1})=v_{i+1}$ for $i = 2 \pmod 3$, $0<i<n$, and $\phi(v_1)=v_1$.
    Thus, $T$ is an edge determining set, since $T$ covers every vertex of $G$ except one. Note $|T| = 2\lfloor \frac{n}{3} \rfloor + 1$,
    so $T$ a minimum edge determining set. Since $n = 2 \pmod 3$, $\lfloor \frac{2n}{3} \rfloor = \lfloor \frac{2(n-2)}{3} + \frac{4}{3} \rfloor = 2 \lfloor \frac{n-2}{3}  \rfloor+ 1=2 \lfloor \frac{n}{3}  \rfloor+ 1$.
    
\end{proof}

The graphs $K_n$ and $K_{n,m}$ illustrate that the determining index can be strictly less than the determining number, sometimes significantly. Further, the example of $K_{n,m}$ when $n > m$  indicates that a minimum edge determining set need not  simply be a cover of the vertices in a minimum vertex determining set. Despite coming close, neither of these families  show the sharpness of the upper bound  $\det(G) \le 2\det'(G)$ in Theorem~\ref{thm:upperbound}. The following example of $K_4-e$ indicates that the bound is sharp.

\input{figure5}

Note there are three nontrivial automorphisms.  One interchanges nonadjacent twin vertices $1$ and $4$, leaving $2$ and $3$ fixed; call it $\alpha$.  Another interchanges adjacent twin vertices $2$ and $3$, leaving $1$ and $4$ fixed; call it $\beta$. A third is the composition of these two automorphisms, $\alpha \circ \beta$, which geometrically is a reflection across a horizontal line through the center of the drawing. Clearly, if we fix only one vertex, then either $\alpha$ or $\beta$ (because the composition moves all four vertices) fixes that vertex but moves others. Hence, we need to fix at least two vertices in order to fix the entire graph, say vertex $1$ and $2$. No nontrivial automorphism  fixes both of these vertices.  However, any automorphism fixing the edge $\{1,2\}$ must fix the graph. By Observation~\ref{obs:difdeg}, no automorphism switches these two vertices as they have different degrees. Further, these vertices constitute a vertex determining set, so this edge constitutes a minimum edge determining set. So $\det(G) = 2 = 2  \det'(G)$. We recall a definition that will allow us to generalize this example to an infinite family of such graphs.

\begin{definition} The {\it join} of graphs $G$ and $H$ is the graph  $G+H$ defined by  $V(G+H)=V(G)\cup V(H)$ and \[E(G+H) = E(G)\cup E(H) \cup \{\{u,v\} \mid  u\in G, v \in H\}.\]
\end{definition}

\begin{theorem}
  For all $n \in \mathbb{N}$, there exists a connected graph $G$ such that \newline $\det(G) = 2n$ and $\det'(G) = n$.
\end{theorem}

\begin{proof}
Let $G = N_{n+1} + K_{n+1}$ for $n \in \mathbb{N}$, where $N_k$ is the empty graph on $k$ vertices. Note that $K_4 - e$ is $N_2 + K_2$. Let $U = \{u_1, \dots , u_{n+1}\} = V(N_{n+1})$ and  $ V = \{v_1, \dots , v_{n+1}\} = V(K_{n+1})$. By construction, the vertices of $U$ are pairwise nonadjacent twins and the vertices of $V$ are pairwise adjacent twins. By Observation~\ref{obs:twins}, there exists an automorphism interchanging them and leaving all other vertices fixed. Thus, a minimum vertex determining set must contain $n$ vertices of $U$ and $n$ vertices of $V$. Since $\deg(u) = n+1$ for $u\in U$ and $\deg(v) = 2n+1$ for $v \in V$, there does not exist an automorphism interchanging the vertices of $U$ with the vertices of $V$. Thus, $\det(G) = 2n$. 

Similarly, an edge determining set must cover $n$ vertices of $U$ and $n$ vertices of $V$. Let $T = \{\{u_i,v_i\}\mid 1 \le i \le n\}$. Because there is no automorphism interchanging the vertices of $U$ and the vertices of $V$, any automorphism fixing the edges of $T$ must fix all their endvertices. Thus, $|T| = \det'(G) = n$. Therefore, $\det(G) = 2\det'(G)$.
\end{proof}

We have seen examples of infinite families such that $\det'(G) =\det(G)$, namely cycles and paths. We find a similar result for trees where the determining number and determining index can be arbitrarily large.  We require the following results, which can be found in any standard text on graph theory, such as \cite{chartrand2012first}.
See Erwin and Harary \cite{ERWIN20063244} for many interesting results on the determining number of trees. 

\begin{proposition} \label{prop:unipaths} \cite{chartrand2012first}
There is a unique path between any two distinct vertices in a tree.
\end{proposition}

\begin{proposition} \label{prop:treecenter}\cite{chartrand2012first}
Every tree has either a single vertex or a pair of adjacent vertices that are contained in every longest path of the tree.
\end{proposition}

\begin{proposition} \label{prop:centerfix}
The center of a tree is fixed by every automorphism. 
\end{proposition}

\begin{proof}
Since path lengths are fixed by automorphisms, the image of any longest path under any automorphism $\phi$  is still a longest path.  So the image of the center under an automorphism must still be in every longest path.  
\end{proof}

\begin{theorem}\label{thm:tree}
  If $G$ is a tree such that $|V(G)| \ge 3$, then $\det(G) = \det'(G)$.
\end{theorem}

\begin{proof}

If $\det(G) = 0$, the result is trivial. So assume $\det(G) \ge 1$. 

By Proposition~\ref{prop:treecenter}, every tree has a center which is either a single vertex or a pair of adjacent vertices.
%that are contained in every longest path of $G$.  
If the center is a pair of vertices, let $v$ be one of them. Otherwise, let $v$ be the center vertex. 

Let $T$ be a minimum edge determining set of a tree $G$. 
For each edge in $T$, one endvertex must be more distant from $v$ than the other by Proposition~\ref{prop:unipaths}. We let $S$ be the set of more distant vertices.  Then $|S| = |T| = \det'(G)$. To show that $S$ is a vertex determining set, 
 assume $\phi\in \aut(G)$ fixes every vertex in $S$. Let $e = \{x,y\} \in T$, with $d(x,v) > d(y,v)$, so that $x \in S$. 
 
 Case 1. If the center is just vertex $v$, then by Proposition~\ref{prop:centerfix}, $\phi(v) = v$.
 
 Case 2: Assume there are two adjacent center vertices, $u$ and $v$. Then again by Proposition~\ref{prop:centerfix}, $\{\phi(u),\phi(v)\}=\{u,v\}$. 
   By assumption $\phi(x) = x$.
 Suppose $\phi(v) = u$. Then $d(x,v) = d(\phi(x),\phi(v))=d(x, u)$.  However, because $u$ and $v$ are adjacent and trees have no cycles, $d(x,u) \neq d(x,v)$. Thus, $\phi(v)=v$. 
 
 Let $w = \phi(y)$. Since $y \in N(x)$ and $\phi(x) = x$, $w\in N(x)$. Since $x$ is more distant from $v$ than $y$, the unique $x-v$ path in the tree $G$ must be $P=(x, y, \dots, v)$. Apply $\phi$ to every vertex in this path to get $P'=(\phi(x), \phi(y), \dots, \phi(v)) = (x, w, \dots, v)$. 
 By Proposition~\ref{prop:unipaths}, there exists a unique path between two vertices, so $P = P'$. Therefore, $\phi(y)=y$. 
 
 We have shown that if $\phi$ fixes only the endvertex of each edge in $T$ that is more distant from $v$, then $\phi$ fixes both endvertices and hence $\phi$ fixes every edge in $T$. Since $T$ is an edge determining set, $\phi$ must be the identity automorphism. Hence, $S$ is a vertex determining set of $G$ of size $\det'(G)$. Hence, $\det(G) \le \det'(G)$.
 By Corollary 1, $\det(G) = \det'(G)$.

\end{proof}

Our last result will be on hypercubes. The $n$-dimensional hypercube, $Q_n$, can be defined as the graph whose vertex set is the set of ordered $n$-bit strings of $0$s and $1$s with two vertices adjacent if they differ in exactly one bit. Hypercubes are highly symmetric graphs; not only are they vertex-transitive and edge-flip-invariant, they are also edge-transitive, arc-transitive and distance-transitive, meaning that given any four vertices $u, v, x$ and $y$ such that $d(u,v) = d(x,y)$, there exists an automorphism $\phi$ such that $\phi(u) = x$ and $\phi(v) = y$. See \cite{Bi1993}. Another definition of hypercubes is based on the binary operation of Cartesian product of graphs.

\begin{definition}
    The {\it Cartesian product} of graphs $G$ and $H$ is the graph  $G\square H$ defined by  $V(G\square H) = \{(u,v) \mid u \in G, v \in H\}$, with   $(u,v)$ and $(x,y)$ adjacent if either $u = x$ and $\{v,y\}\in E(H)$ or $v = y$ and $\{u,x\}\in E(G)$.
\end{definition}

A graph is prime with respect to the Cartesian product if it cannot be written as the Cartesian product of two nontrivial graphs. The prime factor decomposition of a graph is a representation of the graph as a Cartesian product of prime graphs.

We can define $Q_n$ recursively by letting $Q_1 = K_2$ and $Q_n = Q_{n-1} \square K_2$ for $n \ge 2$. Thus $Q_n = K_2 \square \dots \square K_2$ is a prime factor decomposition of the hypercube.
A useful tool in finding the determining index of the hypercube is the characteristic matrix.

\begin{definition}
  Let $S = (V_1, ..., V_r)$ be an ordered set of $m$-tuples. The {\it characteristic matrix}, $M(S)$, is the $r \times m$ matrix with the $ij^{th}$ entry being the $j^{th}$ coordinate of $V_i$.
\end{definition}

The characteristic matrix was used by Boutin \cite{Boutin2009TheDN} to prove the following propositions. 

\begin{proposition}\label{prop:charMatrix} \cite{Boutin2009TheDN}
Let $G$ be a connected graph with prime factor decomposition $G=G_1\square \dots\square G_n$. Let $S = (V_1, \dots, V_n) \subseteq V(G)$. Then $S$ is a vertex determining set if and only if each column of the characteristic matrix, $M(S)$, contains a vertex determining set for the corresponding factor of $G$ and no two columns of $M(S)$ are isomorphic images of each other. 
\end{proposition}

In this proposition, each vertex $V_i$ in the ordered set $S$ is an $n$-tuple $(v_{i1}, v_{i2}, \dots , v_{in})$ with $v_{ij} \in V(G_j)$.  Two columns  of the characteristic matrix, $[v_{1j},\dots,v_{mj}]^T$ and $[v_{1k},\dots,v_{mk}]^T$, are isomorphic if there exists 
a graph isomorphism $\psi: G_j \to G_k$ such that $\psi(v_{ij}) = v_{ik}$ for all $i \in \{1, \dots, n\}$.

\medskip

Applied to $Q_n =  K_2 \square \dots \square K_2$, the characteristic matrix $M(S)$ of a vertex set $S$ will be a $0-1$ matrix.  Since any nonempty subset of vertices of $K_2$ is a vertex determining set, any $0-1$ column is guaranteed to contain a determining set for the corresponding $K_2$ factor of $Q_n$. Moreover, two $0-1$ columns are isomorphic if and only if they have either all identical bits or all opposite bits.

\begin{observation}\label{obs:ZeroOne}
Let $X$ be a $0-1$ matrix with $s$ rows and $t$ columns. There are $2^s$ different $0-1$ vectors of length $s$, which can be partitioned into $2^{s-1}$ opposite pairs. Thus if $t > 2^{s-1}$, then $X$ must have either two columns that are identical or two columns that are opposite.
\end{observation}

This observation is the key to proving the result below.

\begin{proposition}\label{prop:detqn}\cite{Boutin2009TheDN}
 For $n \ge 1$, $\det(Q_n) = \lceil \log_2(n) \rceil + 1$.
\end{proposition}

For $n=1$, $Q_1 = K_2$, which does not have an determining index.  For $n=2$, $Q_2 = C_4$; by Proposition~\ref{prop:detfamilies}, $\det'(C_4) =  \det(C_4) = 2.$

\begin{theorem}\label{thm:QnSC}
  Let $n \ge 3$.
  %If $ n-\lceil \log_2 n \rceil > \lceil 2^{\lceil \log_2 n\rceil -1}$, then $\det'(Q_n) = \lceil \log_2 n  \rceil +1$. Otherwise,  $\det'(Q_n) = \lceil \log_2 \rceil$.
  Then 
  \[
  \det'(Q_n) = \begin{cases} 
  \lceil \log_2 n  \rceil +1, \quad &\text{ if } n-\lceil \log_2 n \rceil >  2^{\lceil \log_2 n\rceil -1},\\
  \lceil \log_2 n  \rceil, & \text{ otherwise}.
  \end{cases}
  \]
\end{theorem} 

\begin{proof}
Assume $T$ is an edge determining set for $Q_n$ with $|T| = \lceil \log_2(n) \rceil - 1$. Let $S$ be the set of endvertices of the edges of $T$. By the proof of Theorem~\ref{thm:upperbound}, $S$ is a vertex determining set. Let $M(S)$ be the corresponding characteristic matrix. By Proposition~\ref{prop:charMatrix}, the columns of $M(S)$ are nonisomorphic. 

Note that the endvertices of each edge in $T$ differ by one bit. Thus there are at most $\lceil \log_2(n) \rceil - 1$ columns such that 
the two rows of $M(S)$ corresponding to the endvertices of a single edge edge of $T$ differ in that columnn.
Let $X$ be the submatrix of $M(S)$ consisting of the other columns,  namely the columns corresponding to bits in which are identical in both endvertices of every edge in $T$. Then $X$ has at least $n - \lceil \log_2(n) \rceil + 1$ columns. Furthermore, $X$ has at most $\lceil \log_2(n) \rceil - 1$ distinct rows, since the endvertices of an edge are identical except for one bit and we removed the columns in which the endvertices differ. 
\medskip

Let $X'$ be the matrix obtained from $X$ by removing any duplicate rows.  Then $X'$ has $s \le \lceil \log_2 n \rceil - 1$ rows and $t \ge n - \lceil \log_2(n) \rceil + 1$ columns. 
We will need the following result, the proof of which is in the appendix.

\begin{lemma}\label{lem:QnAlwaysLess}
For all $n\ge 3$, $n - \lceil \log_2 n \rceil + 1 > 2^{\lceil \log_2 n \rceil - 2}$. 
\end{lemma}

Applying this inequality,
\[
t \ge n - \lceil \log_2(n) \rceil + 1 > 2^{\lceil \log_2 n \rceil - 2} \ge 2^{s-1},
\]
and so by Observation~\ref{obs:ZeroOne},  $X'$ must have some isomorphic columns. Hence so does $X$, and so does $M(S)$. This is a contradiction. Thus $\det'(Q_n) \geq \lceil \log_2(n) \rceil$.

\medskip

 By Corollary~\ref{cor:upperlower}, $\det'(Q_n) \le \det(Q_n) = \lceil \log_2 n \rceil +1$. So the only two options for $\det'(Q_n)$ are $\lceil \log_2 n \rceil$ and $\lceil \log_2 n \rceil + 1$.  We will show that the relative size of $n - \lceil \log_2 n \rceil$ and $ 2^{\lceil \log_2 n \rceil -1}$ decides which option holds.  
 
 \medskip
 
 First assume $n - \lceil \log_2(n) \rceil > 2^{\lceil \log_2(n) \rceil - 1}$. We can repeat the argument given at the beginning of the proof to show that there can be no determining set of size  $\lceil \log_2{n} \rceil$. 
 Using the same logic, we end up with a matrix  $X'$ with $s \le \lceil \log_2 n \rceil$ rows and $t \ge n - \lceil \log_2 n \rceil$ columns. By assumption,
 \[
t \ge n - \lceil \log_2(n) \rceil  > 2^{\lceil \log_2 n \rceil - 1} \ge 2^{s-1},
\]
so the columns of $X'$ cannot be nonisomorphic by Observation~\ref{obs:ZeroOne}.  Hence, in this case, $\det'(Q_n) = \lceil \log_2 n \rceil + 1.$ 
 
 \medskip

Now assume $n - \lceil \log_2 n \rceil \le 2^{\lceil \log_2 n \rceil - 1}$. We will construct an edge determining set $T$ for $Q_n$ of size $\lceil \log_2 n \rceil$. We begin by creating an $\lceil \log_2 n \rceil \times n$ matrix $Y$ whose leftmost $\lceil \log_2 n \rceil$ columns are the standard basis vectors in $\mathbb Z_2^{\lceil \log_2 n \rceil}$ and whose remaining $n - \lceil \log_2 n \rceil$ columns are any set of nonisomorphic $0-1$ columns of length $\lceil \log_2 n \rceil$. For example, when  $n=7$, one possible such matrix is 
\[
Y = \left [\begin{array}{c c c | c c c c} 1 & 0 & 0 & 0 & 1 & 1 & 1 \\ 0 & 1 & 0 & 1 & 0 & 1 & 1\\0 & 0 & 1 & 1 & 1 & 0 & 1 \end{array} \right ].
\]
We then create a corresponding $(2\lceil \log_2 n \rceil) \times n$ matrix $X$ by creating a duplicate copy of each row, then switching only the $1$ that is in the first $\lceil \log_2 n \rceil$ columns to $0$. Continuing with our example, we get
\[
X = \left [\begin{array}{c c c | c c c c} 1 & 0 & 0 & 0 & 1 & 1 & 1 \\ 0 & 0 & 0 & 0 & 1 & 1 & 1 \\ \hline 0 & 1 & 0 & 1 & 0 & 1 & 1\\ 0 & 0 & 0 & 1 & 0 & 1 & 1\\ \hline 0 & 0 & 1 & 1 & 1 & 0 & 1 \\ 0 & 0 & 0 & 1 & 1 & 0 & 1\end{array} \right ].
\]
The first $\lceil \log_2 n\rceil$ columns of $X$ have exactly one $1$. The final $n - \lceil \log_2 n \rceil$ columns of $X$ are still pairwise nonisomorphic, and since each has an even number of $0$' s and an even number of $1$'s, they are also nonisomorphic to any of the first $\lceil \log_2n \rceil$ columns. Hence by Proposition~\ref{prop:charMatrix}, $X$ is the characteristic matrix of a vertex determining set of $Q_n$.

By construction, for each $i \in \{1, \dots , \lceil \log_2 n \rceil\}$, rows $2i-1$ and $2i$ of $X$ differ in exactly one column and therefore they correspond to a pair of adjacent vertices in $Q_n$.  We let $e_i$ be the edge between these vertices. We then let $T= \{e_1, e_2, \dots, e_{\lceil \log_2 n \rceil} \}$. 

\medskip

To show  $T$ is an edge determining set of $Q_n$, suppose $\phi \in \aut(Q_n)$ fixes the edges in $T$.  Then  $\phi$ either fixes or switches the endvertices of each  edge.  To be consistent with the vertex notation in Proposition~\ref{prop:charMatrix}, let $V_i^0$ and $V_i^1$ denote the two endvertices of edge $e_i$, with $V_i^0$ being the endvertex with $0$ in the $i$-th bit.  Let $i$ and $j$ be distinct elements of $\{1, \dots, \lceil \log_2 n \rceil\}$. 
(This is possible because $n \ge 3$, so $\lceil \log_2n \rceil \ge 2$.) 
Suppose the last $n - \lceil \log_2 n \rceil $ columns of the $i$-th and $j$-th rows of $Y$ differ in $\ell$ bits. Using the fact that the distance between two vertices of $Q_n$ is the number of bits in which the two corresponding  bit strings differ, we get
\[
d(V_i^0, V_j^0) = \ell, \quad  d(V_i^0, V_j^1) = d(V_i^1, V_j^0) = \ell+1, \quad  d(V_i^1, V_j^1) = \ell+2.
\]
Since $\phi$ fixes the edges of $T$, $\phi(V_i^0) = V_i^0$ or $V_i^1$ and $\phi(V_j^0) = V_j^0$ or $V_j^1$. Because automorphisms are distance-preserving, it must be the case that 
\[
\ell = d(V_i^0, V_j^0) = d(\phi(V_i^0), \phi(V_j^0)).
\]
This is only possible if $\phi$ fixes the endvertices of both $e_i$ and $e_j$. Thus $\phi$ must fix the endvertices of every edge in $T$.  As noted earlier, these vertices constitute a vertex determining set of $Q_n$ and so $\phi$ must be the identity. Hence $T$ is an edge determining set.
\end{proof}

An interesting implication of Theorem~\ref{thm:QnSC} is that as $n$ increases, $\det'(Q_n)$ fluctuates infinitely many times between the upper bound of $\det(Q_n)$ and the marginally smaller value of $\det(Q_n) - 1$. The graphs in Figure~\ref{fig:GraphComparison} illustrate that  $n - \lceil \log_2 n \rceil$ and $2^{\lceil \log_2 n \rceil - 1}$ keep alternating in relative size as $n$ increases. An alternative formulation of this result may clarify why this happens. 

\begin{figure} 
\center
\includegraphics[scale=.75]{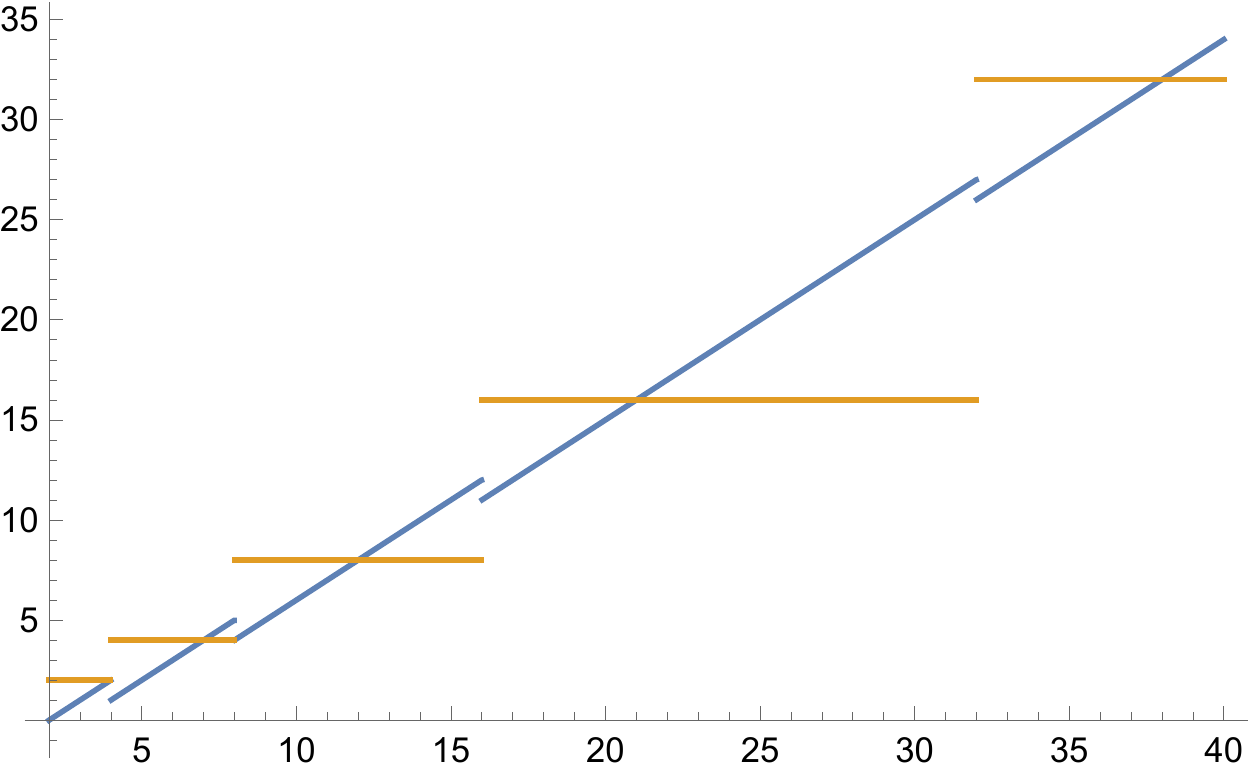}
\caption{Comparing $x- \lceil \log_2 x \rceil$ and $2^{\lceil \log_2 x \rceil -1}$.}
\label{fig:GraphComparison}
\end{figure}

%If we let $k = \lceil \log_2 n \rceil$, then $2^{k-1} < n \le 2^k$. Then $n - \lceil \log_2 n \rceil > 2^{\lceil \log_2 n \rceil -1}$ if and only if $n - k > 2^{k-1}$. 

\begin{corollary}
Let $n \ge 3$ and let $k = \lceil \log_2 n \rceil$. Then
\[
  \det'(Q_n) = \begin{cases} 
  k, & \text{ if } 2^{k-1} < n \le 2^{k-1} +k, \\
  k +1, \quad &\text{ if } 2^{k-1} +k < n \le 2^k.
  
  \end{cases}
  \]
\end{corollary}

\begin{proof}
Since $k\le 2^{k-1}$ for all $k \in \mathbb N$, both $n$ and $2^{k-1} + k$ lie in the interval $(2^{k-1}, 2^k]$.
Either they are equal, or one is to the left of the other.
If $n \le 2^{k-1} + k$, then 
$
n - \lceil \log_2 n \rceil = n - k \le 2^{k-1} = 2^{\lceil \log_2 n \rceil -1},
$
and so $\det'(Q_n) = \lceil \log_2 n \rceil$ by Theorem~\ref{thm:QnSC}. 
By similar reasoning, if  $n > 2^{k-1} + k$, then 
%$n - \lceil \log_2 n \rceil = n - k > 2^{k-1} = 2^{\lceil \log_2 n \rceil -1},$and so 
$\det'(Q_n) = \lceil \log_2 n \rceil + 1$. 
\end{proof}

\section{Open Questions}

 In this paper, we have provided the determining index for several families of graphs, including paths, cycles, complete graphs, complete bipartite graphs, trees and hypercubes. There are many more families of symmetric graphs for which this parameter could be computed, such as generalized Petersen graphs, circulant graphs, orthogonality graphs, Mycielskian graphs, Paley graphs and Praeger-Xu graphs, to name a few. In addition, there are some more conceptual open questions.

\begin{enumerate}
  \item We have shown that in some cases, the determining index of a graph is strictly smaller than the determining number. In these cases, the number of edges that are needed to fix in order to break all symmetries of the graph is less than the number of vertices that would be required, suggesting that fixing edges is more efficient. Are there necessary and/or sufficient conditions on $G$ that guarantee $\det(G)=\det'(G)$?
  \item By the proof of Theorem~\ref{thm:QnSC}, the edge set 
  \[T = \{\{(1,0,0,0),(0,0,0,0)\},\{(0,1,1,0),(0,0,1,0)\}\}\]
  is a minimum edge determining set for $Q_4$. The set of endvertices of the edges of $T$ is a vertex determining set but not minimum by Proposition~\ref{prop:detqn}.  The characteristic matrices of all possible subsets with one end vertex removed are
  \[
\left [\begin{array}{c c c c} 1 & 0 & 0 & 0 \\ 0 & 0 & 0 & 0 \\0 & 1 & 1 & 0  \end{array} \right ],
\left [\begin{array}{c c c c} 1 & 0 & 0 & 0 \\ 0 & 0 & 0 & 0 \\0 & 0 & 1 & 0  \end{array} \right ],
\left [\begin{array}{c c c c} 0 & 0 & 0 & 0 \\ 0 & 1 & 1 & 0 \\0 & 0 & 1 & 0  \end{array} \right ],
\left [\begin{array}{c c c c} 1 & 0 & 0 & 0 \\ 0 & 1 & 1 & 0 \\0 & 0 & 1 & 0  \end{array} \right ].
\]
  Each of these characteristic matrices has two isomorphic columns, so by Proposition~\ref{prop:charMatrix} they are not vertex determining sets of $Q_4$. Since $\det(Q_4) = 3$, we have found a minimum edge determining set whose set of endvertices does not contain a minimum vertex determining set. Do there exist conditions such that the set of endvertices of the edges in a minimum edge determining set necessarily contains a minimum vertex determining set? 
  \item Are there expressions for $\det'(G+H)$ and  $\det'(G\square H)$ in terms of $\det'(G)$ and $\det'(H)$? 
 
\end{enumerate}

\begin{appendices} 
\section{Proof of Lemma~\ref{lem:QnAlwaysLess}}

{\bf Lemma~\ref{lem:QnAlwaysLess}.}
For all $n \ge 3$, $2^{\lceil \log_2 n \rceil - 2} < \frac{n}{2} \le n - \lceil \log_2 n \rceil +1$.

\begin{proof}
By definition of the ceiling function,  $\lceil \log_2 n \rceil - 1 < \log_2 n \le \lceil \log_2 n \rceil$ for any $n \in \mathbb N$. Then $2^{\lceil \log_2 n \rceil - 1} < 2^{\log_2 n} = n$, because exponential functions are increasing. 
Dividing  by $2$ gives $2^{\lceil \log_2 n \rceil - 2} < \frac{n}{2}$. 

\medskip

Next, let $f(x) = \log_2 x - \frac{x}{2}$, for $0 < x \in \mathbb R$. It is easily verified that the only zeroes of $f$ occur at $x = 2$ and $x=4$. Elementary calculus can be used to show that the graph $y = f(x)$ is concave down, with the absolute maximum value at $x = \frac{2}{\ln 2} \approx 2.885$. This means that for all $x \ge 4$, $f(x) \le 0$, which implies $\log_2 x \le \frac{x}{2}$. Hence, for all $n \ge 4$
\[
\lceil \log_2 n \rceil - 1 < \log_2 n \le \frac{n}{2},
\]
which can be algebraically rearranged to 
\[
\frac{n}{2} \le n - \lceil \log_2 n \rceil +1.
\]
A direct calculation shows that this inequality also holds for $n = 3$.
\end{proof}

\end{appendices}

\newpage

\bibliography{References}{}
\bibliographystyle{plain}

\end{document}

%% file: figure1.tex
% FIGURE 1:
\begin{figure}[h]
\begin{subfigure}{0.3\textwidth}
\centering
\begin{tikzpicture}[main/.style = {draw, circle}] 
  \node[main] (a1) at (0,0) {1};  
  \node[main] (a2) at (2,0)  {2}; 
  \node[main] (a3) at (2,2)  {3};  
  \node[main] (a4) at (0,2) {4};  
  
  \draw (a1) -- (a2); 
  \draw (a2) -- (a3);  
  \draw (a2) -- (a4); 
  \draw (a3) -- (a4);
\end{tikzpicture}
\caption{$G_1$}
\end{subfigure}
\begin{subfigure}{0.3\textwidth}
\centering
\begin{tikzpicture}[main/.style = {draw, circle}] 
 \node[main] (a1) at (1,1) {1};  
  \node[main] (a2) at (3,1)  {2}; 
  \node[main] (a3) at (3,3)  {3};  
  \node[main] (a4) at (1,3) {4};

  \draw (a1) -- (a2); 
  \draw (a2) -- (a3);  
  \draw (a3) -- (a4); 
  \draw (a1) -- (a3);
  \draw (a2) -- (a4);
\end{tikzpicture}
\caption{$G_2$}
\end{subfigure}
\begin{subfigure}{0.3\textwidth}
\centering
\begin{tikzpicture}[main/.style = {draw, circle}] 
  \node[main] (a1) at (0,0) {1};  
  \node[main] (a2) at (2,0)  {2}; 
  \node[main] (a3) at (2,2)  {3};  
  \node[main] (a4) at (0,2) {4};  
  
  \draw (a1) -- (a2); 
  \draw (a2) -- (a3);  
  \draw (a2) -- (a4); 
  \draw (a1) -- (a4);
  \draw (a1) -- (a3);
  \draw (a3) -- (a4);
\end{tikzpicture}
\caption{$G_3$}
\end{subfigure}
\caption{Exceptions to $\aut(G)\cong \aut(L(G))$.}
\label{fig:specialcase}
\end{figure}

%% file: line_graphs.tex
\begin{figure}[h]
\begin{subfigure}{0.3\textwidth}
\centering
\begin{tikzpicture}[main/.style = {draw, circle}] 
  \node[main] (a1) at (1,1) {$1,2$};  
  \node[main] (a2) at (3,1)  {$2,4$}; 
  \node[main] (a3) at (3,3)  {$2,3$};  
  \node[main] (a4) at (1,3) {$3,4$};
  
  \draw (a1) -- (a2); 
  \draw (a2) -- (a3);  
  \draw (a3) -- (a4); 
  \draw (a1) -- (a3);
  \draw (a2) -- (a4);
\end{tikzpicture}
\caption{$L(G_1)$}
\end{subfigure}
\begin{subfigure}{0.3\textwidth}
\centering
\begin{tikzpicture}[main/.style = {draw, circle}] 
 \node[main] (a1) at (1,1) {$2,4$};  
  \node[main] (a2) at (3,1)  {$3,4$}; 
  \node[main] (a3) at (3,3)  {$2,3$};  
  \node[main] (a4) at (1,3) {$1,3$};
  \node[main] (a5) at (2,4) {$1,2$};

  \draw (a1) -- (a2); 
  \draw (a2) -- (a3);  
  \draw (a3) -- (a4); 
  \draw (a1) -- (a3);
  \draw (a2) -- (a4);
  \draw (a5) -- (a4);
  \draw (a5) -- (a1);
  \draw (a5) -- (a3);
\end{tikzpicture}
\caption{$L(G_2)$}
\end{subfigure}
\begin{subfigure}{0.3\textwidth}
\centering
\begin{tikzpicture}[main/.style = {draw, circle}] 
  \node[main] (a1) at (0,0) {$1,4$};  
  \node[main] (a2) at (2.5,0)  {$1,3$}; 
  \node[main] (a3) at (2.5,2)  {$2,3$};  
  \node[main] (a4) at (0,2) {$2,4$};
  \node[main] (a5) at (1.25,3) {$1,2$};
  \node[main] (a6) at (1.25,-1) {$3,4$};
  
  \draw (a1) -- (a2); 
  \draw (a2) -- (a3);  
  \draw (a2) -- (a6);
  \draw (a1) -- (a6); 
  \draw (a1) -- (a4);
  \draw (a6) -- (a3);
  \draw (a6) -- (a4);
  \draw (a3) -- (a4);
  \draw (a3) -- (a5);
  \draw (a4) -- (a5); 
  \draw (a1) -- (a5);
  \draw (a2) -- (a5);
\end{tikzpicture}
\caption{$L(G_3)$}
\end{subfigure}
\caption{Line Graphs to the Exceptions}
\label{fig:linegraphs}
\end{figure}

%% file: figure2.tex
\begin{figure}[h]
\begin{subfigure}{0.5\textwidth}
\centering
\begin{tikzpicture}[main/.style = {draw, circle}] 
  \node[main] (a1) at (0,0) {1};  
  \node[main] (a2) at (1.5,0)  {2}; 
  \node[main, fill=red] (a3) at (3,0)  {3};  
   \node[main] (a4) at (4.5,0)  {4}; 
  
  \draw (a1) -- (a2); 
  \draw[color=blue, line width=3pt] (a2) -- (a3);  
  \draw (a3) -- (a4); 
\end{tikzpicture}
\caption{$P_4$}
\label{fig:P_4a}
\end{subfigure}
\begin{subfigure}{0.5\textwidth}
\centering
\begin{tikzpicture}[main/.style = {draw, circle}] 
  \node[main] (a1) at (0,0) {1};  
  \node[main] (a2) at (1.5,0)  {2}; 
  \node[main, fill=red] (a3) at (3,0)  {3};  
   \node[main] (a4) at (4.5,0)  {4}; 
  
  \draw (a1) -- (a2); 
  \draw (a2) -- (a3);  
  \draw[color=blue, line width=3pt] (a3) -- (a4); 
\end{tikzpicture}
\caption{$P_4$}
\label{fig:P_4b}
\end{subfigure}
\caption{Comparing $\det(G)$ and $\det'(G)$ on $P_4$.}
\end{figure}

%% file: figure3.tex
\begin{figure}[h]
\begin{subfigure}{0.5\textwidth}
\centering
\begin{tikzpicture}[main/.style = {draw, circle}] 
  \node[main] (a1) at (0,0) {1};  
  \node[main] (a2) at (4,0)  {2}; 
  \node[main, fill=red] (a3) at (4,4)  {3};  
  \node[main] (a4) at (0,4) {4};
  \node[main, fill=red] (a5) at (3,2)  {5};  
  \node[main] (a6) at (1,2) {6}; 
  
  \draw (a1) -- (a2); 
  \draw (a2) -- (a3);  
  \draw[color=blue, line width=3pt] (a3) -- (a4); 
  \draw (a1) -- (a4);

  \draw[color=blue, line width=3pt] (a5) -- (a6); 
 
  \draw (a3) -- (a5); 
  \draw (a4) -- (a6);
  \draw (a1) -- (a6); 
  \draw (a2) -- (a5);
\end{tikzpicture}
\caption{$H$}
\label{fig:enva}
\end{subfigure}
\begin{subfigure}{0.5\textwidth}
\centering
\begin{tikzpicture}[main/.style = {draw, circle}] 
  \node[main] (a1) at (0,0) {1};  
  \node[main] (a2) at (4,0)  {2}; 
  \node[main, fill=red] (a3) at (4,4)  {3};  
  \node[main] (a4) at (0,4) {4};
  \node[main, fill=red] (a5) at (3,2)  {5};  
  \node[main] (a6) at (1,2) {6}; 
  
  \draw (a1) -- (a2); 
  \draw (a2) -- (a3);  
  \draw[color=blue, line width=3pt] (a3) -- (a4); 
  \draw (a1) -- (a4);

  \draw (a5) -- (a6); 
 
  \draw[color=blue, line width=3pt] (a3) -- (a5); 
  \draw (a4) -- (a6);
  \draw (a1) -- (a6); 
  \draw (a2) -- (a5);
\end{tikzpicture}
\caption{$H$}
\label{fig:envb}
\end{subfigure}
\caption{Comparing $\det(H)$ and $\det'(H)$}
\label{fig:env}
\end{figure}

%% file: figure_neighbor_ex.tex
\begin{figure}[h]

\centering
\begin{tikzpicture}[main/.style = {draw, circle}] 
  \node[main] (a1) at (1,1) {$z$};  
  \node[main] (a2) at (3,1)  {$x$}; 
  \node[main] (a3) at (3.5,0)  {$y$};  
  \node[main] (a4) at (2,2) {$v$};
  \node[draw=none] (a5) at (2,1) {};
  \node[draw=none] (a6) at (1,0) {};
  \node[draw=none] (a7) at (.5,0) {};
  \node[draw=none] (a8) at (1.5,0) {};
  \node[draw=none] (a9) at (3,0) {};
  \node[draw=none] (a10) at (2.5,0) {};

  \draw (a1) -- (a4); 
  \draw (a2) -- (a3);  
  \draw (a2) -- (a4);
  \draw[dotted] (a4) -- (a5);
  \draw[dotted] (a1) -- (a6);
  \draw[dotted] (a1) -- (a7);
  \draw[dotted] (a1) -- (a8);
  \draw[dotted] (a2) -- (a9);
  \draw[dotted] (a2) -- (a10);

\end{tikzpicture}
\caption{$x\in N(v)$ and $y \in N(x)$, $y \neq v$}
\label{fig:neighborneighbor}

\end{figure}

%% file: figure4.tex
\begin{figure}[h]

\centering
\begin{tikzpicture}[main/.style = {draw, circle}, scale=0.75] 

  \node[main] (a1) at (0.0, 4.0) {0};
\node[main] (a2) at (2.57, 3.06) {1};
\node[main] (a3) at (3.93, 0.69) {2};
\node[main] (a4) at (3.46, -1.99) {3};
\node[main] (a5) at (1.36, -3.75) {4};
\node[main] (a6) at (-1.36, -3.75) {5};
\node[main] (a7) at (-3.46, -2.00) {6};
\node[main] (a8) at (-3.93, 0.69) {7};
\node[main] (a9) at (-2.57, 3.06) {8};
% \node[main] (a10) at (1.36, 3.75) {[0, 1, 3]};
% \node[main] (a11) at (3.46, 2.00) {[1, 2, 4]};
% \node[main] (a12) at (3.93, -0.69) {[2, 3, 5]};
% \node[main] (a13) at (2.57, -3.06) {[3, 4, 6]};
% \node[main] (a14) at (0, -4.0) {[4, 5, 7]};
% \node[main] (a15) at (-2.57, -3.06) {[5, 6, 8]};
% \node[main] (a16) at (-3.93, -0.69) {[6, 7, 0]};
% \node[main] (a17) at (-3.46, 1.99) {[7, 8, 1]};
% \node[main] (a18) at (-1.36, 3.75) {[8, 0, 2]};
\node[main] (a10) at (1.36, 3.75) {9};
\node[main] (a11) at (3.46, 2.00) {10};
\node[main] (a12) at (3.93, -0.69) {11};
\node[main] (a13) at (2.57, -3.06) {12};
\node[main] (a14) at (0, -4.0) {13};
\node[main] (a15) at (-2.57, -3.06) {14};
\node[main] (a16) at (-3.93, -0.69) {15};
\node[main] (a17) at (-3.46, 1.99) {16};
\node[main] (a18) at (-1.36, 3.75) {17};
\node[scale=1.5] (a19) at (-.88, 5.3) {$\beta$};
\node[scale=1.5] (a20) at (-3.8, 3.2) {$\alpha$};
\node[scale=1.5] (a20) at (.88, 5.3) {$\gamma$};

  \draw (a1) -- (a10); 
  \draw (a1) -- (a18);
  \draw (a1) -- (a16);
  \draw (a2) -- (a10);
  \draw (a2) -- (a11); 
  \draw (a2) -- (a17); 
  \draw (a3) -- (a11); 
  \draw (a3) -- (a18);
  \draw (a3) -- (a12);
  \draw (a4) -- (a10);
  \draw (a4) -- (a12); 
  \draw (a4) -- (a13); 
  \draw (a5) -- (a11); 
  \draw (a5) -- (a13);
  \draw (a5) -- (a14);
  \draw (a6) -- (a14);
  \draw (a6) -- (a15); 
  \draw (a6) -- (a12); 
  \draw (a7) -- (a15); 
  \draw (a7) -- (a16);
  \draw (a7) -- (a13);
  \draw (a8) -- (a14);
  \draw (a8) -- (a16); 
  \draw (a8) -- (a17); 
  \draw (a9) -- (a15);
  \draw (a9) -- (a17); 
  \draw (a9) -- (a18); 
  
  \draw[red,thick,dashed] (.85,-5) -- (-.85,5);
  \draw[red,thick,dashed] (3.6,-3) -- (-3.6,3);
  \draw[red,thick,dashed] (-.85,-5) -- (.85,5);
 
\end{tikzpicture}
\caption{$G_4$}
\label{fig:BrooksGraph}

\end{figure}

%% file: figure5.tex
\begin{figure}[h]

\centering
\begin{tikzpicture}[main/.style = {draw, circle}] 
  \node[main] (a1) at (1,1) {1};  
  \node[main] (a2) at (3,1)  {2}; 
  \node[main] (a3) at (3,3)  {3};  
  \node[main] (a4) at (1,3) {4};

  \draw (a1) -- (a2); 
  \draw (a2) -- (a3);  
  \draw (a3) -- (a4); 
  \draw (a1) -- (a3);
  \draw (a2) -- (a4);

\end{tikzpicture}
\caption{$K_4 - e$}

\end{figure}